\documentclass{amsart}

\usepackage{amsmath,amssymb,amsfonts,amsthm,color}
\usepackage{eucal}
\usepackage{ stmaryrd }
\usepackage{tikz}
\usepackage{verbatim}

\usetikzlibrary{matrix, arrows}

\usepackage{graphicx,psfrag}

\usepackage{multicol}

\usepackage{ifthen} 

\newcommand{\baseenvskip}{\baselineskip 6.53mm}

\newtheorem{thm}{Theorem}

\newtheorem{claim}[thm]{Claim}

\newtheorem{prop}[thm]{Proposition}

\newtheorem{corr}[thm]{Corollary}

\theoremstyle{remark}
\newtheorem{rmk}{Remark}[section]

\theoremstyle{remark}
\newtheorem{ex}{\rm{\textbf{Example}}}[section]

\newtheorem{dfn}{\rm{\textbf{Definition}}}[section]
\numberwithin{equation}{section} \numberwithin{thm}{section}
\numberwithin{rmk}{section} \numberwithin{figure}{section} \numberwithin{dfn}{section}
\numberwithin{ex}{section}

\newcommand{\ds}{\displaystyle}

\newcommand{\R}{\mathbb{R}}

\newcommand{\Z}{\mathbb{Z}}

\newcommand{\bd}{\partial}
\newcommand{\td}{\tilde}

\newcommand{\mbb}{\mathbb}
\newcommand{\Fron}{\text{Fron}}
\newcommand{\spt}{\text{spt}}

\newcommand{\llb}{\llbracket}
\newcommand{\rrb}{\rrbracket}
\newcommand{\cross}{\times}

\newcommand{\mass}{\mbb{M}}
\newcommand{\warrow}{\rightharpoonup}

\newcommand{\restr} {
\hskip2.5pt{\vrule height7pt width.5pt depth0pt}
\hskip-.2pt\vbox{\hrule height.5pt width7pt depth0pt}
\, }

\title[ Semi-Algebraic Homology and Mass Minimization ]{Homology Classes of  Semi-Algebraic Sets and Mass Minimization}

\author{Quentin Funk}\thanks{Partially supported by NSF VIGRE DMS-1148609}
    \email{qfunk@rice.edu}
\begin{document}

\begin{abstract}
We associate to any compact semi-algebraic set $X \subset \R^n$ a chain complex of currents $S_\ast (X)$ generated by integration along semi-algebraic submanifolds and we analyze the corresponding homology groups. In particular, we show that these homology groups satisfy the Eilenberg-Steenrod axioms and further, that they are isomorphic to both the ordinary singular homology groups of $X$ and to the homology groups generated by the integral currents supported on $X$.  Using this result and a certain neighborhood of $X$, we are able to prove homological mass minimization for integral currents supported on $X$, and verify that any cycle of $X$ that has sufficiently small mass is a boundary.
\end{abstract}

\maketitle

\section{Introduction}

One is often interested in finding special representatives of homology classes. For compact Riemannian manifolds, the existence of length-minimizing geodesics (possibly with multiplicity) in one dimensional integral homology classes raised the question of finding mass-minimizers in higher dimensional classes. Despite possible obstructions to homology classes having manifold representatives (demonstrated in \cite{Thom}) this question was resolved in 1960 by H. Federer and W. Fleming \cite{FFCurrents}--they obtained minimizers as certain weak limits of manifolds called integer-multiplicity rectifiable currents.

Similarly, they defined a homology theory via rectifiable currents as follows: if $I_k(A)$ denotes the group of normal (finite mass, finite boundary mass) $k$-dimensional integer-multiplicity rectifiable currents in $\mbb{R}^n$ (hereafter referred to as integral currents) supported on $A$, then the associated homology groups are:

\begin{equation*}
H_k^{IC}(A,B)\ =\dfrac{\{ T\in {I}_k(A)\ :  \partial T\in  I_{k-1}( B)\}}{\{R+\partial S\ :\ R\in  I_k(B), \ S\in  I_{k+1}(A)\}}.
\end{equation*}

\noindent Here the superscript $IC$ denotes the fact that these groups stem from arbitrary integral currents since we will later introduce homology groups defined by a subgroup of integral currents.

For pairs of compact Riemannian manifolds, or more generally Euclidean Lipschitz neighborhood retracts (ELNR), they  verified that integral currents gave homology isomorphic to the ordinary singular homology with coefficients in $\mbb Z$. This isomorphism combined with an isoperimetric inequality gives (in the compact case) a mass minimizing rectifiable representative for every ordinary homology class. For details, see \cite[4.4.1, 5.1]{GMT}.

One can ask the same questions for algebraic varieties or, more concretely, for semi-algebraic subsets of $\mbb{R}^n$: sets  defined by finitely many polynomial equalities and inequalities (see, for example, \cite[$\S$ 2]{RAG}). There is much interest in semi-algebraic spaces since they are an integral part of the study of real algebraic geometry and an example of an $o$-minimal structure, see \cite{RAG} and \cite{TTOS}, respectively.  Simple examples of compact semi-algebraic sets have cusps and fail to be Euclidean Lipschitz neighborhood retracts, however, they are homeomorphic to finite polyhedra so singular theory gives a natural homology theory that is uniquely characterized by the Eilenberg-Steenrod axioms.

 In this paper we verify that, for a pair $B\subseteq A$ of semi-algebraic subsets of $\mbb{R}^n$,  the ordinary singular homology groups are isomorphic with the integral current homology groups, and use this fact to generalize two classical results concerning rectifiable currents from the ELNR case to the semi-algebraic case, namely,  that for any compact semi-algebraic set,  there is a mass-minimizing rectifiable respresentative in each of its finitely many homology classes and we are able to prove (again, similarly to the ELNR case, seen in \cite[9.6]{FFCurrents}) that cycles of small enough mass in a compact semi-algebraic set are boundaries.

It is convenient to use other homology groups to verify these facts. Our main result is to establish an isomorphism between the rectifiable current homology groups and the more well-behaved homology given by the much smaller subclass of semi-algebraic chains \cite{SRAMI, TPSS}: those integral currents whose support and boundary support are semi-algebraic sets. Further, semi-algebraic maps (i.e. continuous maps with semi-algebraic graphs) are suitable morphisms for semi-algebraic chains--even though they may fail to be locally Lipschitz. Substituting semi-algebraic sets  and maps into the definitions above, we obtain a homology theory that satisfies the Eilenberg-Steenrod Axioms (see Section \ref{sec:Eilenberg}). Therefore, in the compact case, the  semi-algebraic homology coincides with ordinary singular theory. We appreciate the question about this of Blaine Lawson, which stimulated this work.

The main idea for this paper comes from an attempt to mimic the proof of the mass minimization result for compact Riemannian manifolds, proven by H. Federer in \cite{GMT}. The main point of that proof was that any such manifold could be isometrically embedded in Euclidean space as a Lipschitz retract of some of its open neighborhoods on which tools such as the deformation theorem could be applied. Unfortunately, a semi-algebraic set need not be a Lipschitz retract of \emph{any} of its open neighborhoods--as can be seen by simple examples such as $\{(t^2, t^3) | t\in \R\} \subset \R^2$. 

However, in loose analogy with the ELNR case, given a semi-algebraic $X\subset \R^n$, there are arbitrarily small neighborhoods which are semi-algebraically contractible--and thus the inclusion of $X$ into these neighborhoods induces a surjective chain map between the groups of semi-algebraic chains, and it is this surjection that enables us to generalize the other results to this context once the above isomorphism is established. 

Since each semi-algebraic chain is an integral current, proving the isomorphism reduces to showing the following is true in (pairs of) compact semi-algebraic sets:

\begin{enumerate}
\item Every semi-algebraic cycle that bounds an integral chain also bounds a semi-algebraic chain.
\item Every rectifiable cycle is homologous to a semi-algebraic cycle.
\end{enumerate}

\noindent We verify these facts in Section \ref{sec:MT}, see Propositions \ref{prop:inj} and \ref{prop:surj}, respectively. (1) follows from retracting certain deformations of semi-algebraic cycles supported on $X$ within the open neighborhoods described above. (2) is more difficult and relies on the local Lipschitz contractibility of semi-algebraic sets, as shown by L. Shartser and G. Valette in \cite{deRham} (see also the independent work of L. Shartser, \cite{Shartserthesis}).

Related topics are also addressed in the interesting paper of T. De Pauw  \cite{Comphom}  which treats relations among integral currents, various homology theories, and Plateau problems. Similarly, integral current homology for chains in metric spaces is discussed in \cite{DHP}.

The main reference for the geometric measure theory contained in the following is H. Federer's treatise \cite{GMT}, but much of the theory can be found in the books of F. H. Lin and X. Yang, and the work of L. Simon, \cite{Lin}, \cite{Simon}, respectively. Similarly, we repeatedly use well-known properties of semi-algebraic sets, such as triangulability and the existence of stratifications. To the author's knowledge, the most comprehensive reference for this material is the book of J. Bochnak, M. Coste and M. Roy, \cite{RAG}. Similar results for the more general case of $o$-minimality may be found in the book of L. van den Dries, \cite{TTOS}. See also the paper of M. Edmundo and A. Woerheide \cite{Edmundo}, where classical homology theories were analyzed in this more general setting.

\section{Semi-Algebraic Chains: Definitions and Elementary Properties}
\label{sec:elem}

We denote by $\mathcal{D}^k(\mbb{R}^n)$ the set of compactly supported differential $k$-forms on $\mbb{R}^n$, and by $\mathcal{H}^\alpha$ the $\alpha$-dimensional Hausdorff measure on $\mbb{R}^n$. Finally, for a semi-algebraic subset $X$ of $\R^n$, we denote $\Fron(X) = \overline X \setminus X$ the \emph{frontier} of $X$.

If $f:M\subset \mbb{R}^m \rightarrow \mbb{R}^n$ is a semi-algebraic map with $M$ a semi-algebraic subset of $\mbb{R}^m$, then by \cite[2.5]{HardtHomotopy} we may find a finite collection $\{V_i\}_{i=1}^q$ of subsets of $M$ so that $M=\ds \cup V_i$ and so that the following holds:
\begin{enumerate}
\item each $V_i$ is a connected, real analytic, embedded submanifold of $\mbb{R}^n$ with $\Fron(V_i)$ being a union of lower dimensional elements of $\{V_i\}_{i=1}^q$.
\item $f(V_i)$ is a real analytic, semi-algebraic manifold and either $f|_{V_i}$ is a real analytic diffeomorphism onto $f(V_i)$, or $f|_{V_i}$ is real analytic and of constant rank less than $\dim(V_i)$.
\end{enumerate}

\noindent Such a partition of $M$ is known as a \emph{stratification} of $M$ and we call the elements of $\{V_i\}_{i=1}^q$ the \emph{strata}. We may also require that $\{f(V_i)\}_{i=1}^q$ is a stratification for $f(M)$ and further, if $A\subset M$ is a semi-algebraic subset, that $A$ is a union of some subcollection of $\{V_i\}_{i=1}^q$. A stratification satisfying this last property is said to be \emph{compatible} with $A$. 

Now suppose that $M$ is an oriented, compact, $k$-dimensional smooth semi-algebraic submanifold of $\mbb{R}^n$. Relabeling if necessary, let $\{V_i\}_{i=1}^\ell$ denote the strata such that $f|_{V_j}$ is of rank $k$. Set $\mathcal{N}=\{f(V_1)\}_{i=1}^\ell$, and choose the orientation on each $f(V_i)$ induced by pushing forward the orientation of $V_i$ inherited from the orientation of $M$ by $f|_{V_i}$.   Denote the inclusion map $f(V_i)\hookrightarrow \mbb{R}^n$ by $\phi_{V_i}$ and define a current $f_{\sharp} (\llb M \rrb)$ such that:

\begin{equation}
\label{eq:semichain}
f_{\sharp} (\llb M \rrb)(\omega)=\ds \sum_{i=1}^\ell  \int_{f(V_i)} \phi_{V_i}^\ast (\omega)
\end{equation}

\noindent for all $\omega \in \mathcal{D}^k(\mathbb{R}^n)$.

Even though $f$ is not Lipschitz, $f(V_i)$ has finite $\mathcal{H}^k$ measure by \cite[3.4.8 (3)]{GMT} and by Proposition \ref{prop:functoriality}, we see that any such current is an integral current.

 By combining terms where the sets $f(V_i)$ and $f(V_j)$ are equal, we may simplify the expression to:

\begin{equation}
\label{eq:semichain2}
f_{\sharp} (\llb M \rrb)(\omega)=\ds \sum_{N\in\mathcal{N}} n_N  \int_{N} \phi_{N}^\ast (\omega)
\end{equation}

\noindent where each integer $n_N$ depends on a choice of orientation of $N$ and whether each $f|_{V_i}$  preserves or reverses the orientation. If $\nu_N$ is an orientation form for $N\in\mathcal{N}$ consistent with the orientation chosen (i.e. a unit simple $k$-vectorfield on $N$ whose vectors span $T_x N$ at every point $x\in N$), then we have the alternate representation:

\begin{equation*}
f_{\sharp} (\llb M \rrb)(\omega)=\ds \sum_{N\in\mathcal{N}} n_N\int_N \phi_{N}^{\ast} (\omega)(x)(\nu_N (x) ) \, d \mathcal{H}^k (x).
\end{equation*}

 Our next goal is to show that the integral above may be re-written by the area formula as an integral over $V_i$ of the pullback form $(f|_{V_i})^\ast (\phi_{V_i}^\ast (\omega))$.
 
 Note that since $f|_{V_i}$ is smooth, for any compact exhaustion $\{V_{i,j}\}_{j\in \mbb N}$ of $V_i$, we have, for all $j$ that $f|_{V_{i,j}}$ has bounded derivative and is therefore Lipschitz. Therefore, denoting the inclusion $f(V_{i,j})\hookrightarrow \R^n$ by $\phi_{V_{i,j}}$, the area formula implies that:

\begin{equation}
\int_{V_{i,j}} J_k df|_{V_{i,j}}(x)  \, d \mathcal{H}^k (x) = \mathcal H^k(f(V_{i,j}) \leq \mathcal H^k (f(V_i)) <\infty.
\end{equation}

\noindent Therefore by applying the Monotone Convergence Theorem as $j\rightarrow \infty$, we obtain that the Jacobian $J_k df|_{V_{i}}$ is integrable over $V_i$. Next notice that

\begin{equation}
\int_{f(V_{i,j})}\phi_{V_{i,j}}^\ast (\omega)(x) (\nu_{f(V_{i,j})})  \, d \mathcal{H}^k (x) = \int_{V_{i,j}} (f|_{V_{i,j}})^\ast (\phi_{V_{i,j}}^\ast (\omega))(\nu_{V_{i,j}})  \, d \mathcal{H}^k (x). 
\end{equation}

\noindent Since $\phi$ is smooth on $\R^n$ and thus bounded in some neighborhood of $X$, the integrability of $J_k df|{V_{i}}$ allows us to apply the Dominated Convergence Theorem (again, as $j\rightarrow \infty$) to the above equality. Summing over $i$ then gives:

\begin{equation}
f_{\sharp} (\llb M \rrb)(\omega) = \sum_i \int_{V_i} (f|_{V_i})^\ast (\phi_{V_i}^\ast (\omega)).
\end{equation}

\noindent The above equality implies that our definition is invariant under refinement of the stratification.

Therefore, if we are given a different stratification of $M$, say $\{W_i\}_{i=1}^{\td \ell}$, by finding a common refinement for the two stratifications we see that the rule $f_\sharp (\llb M \rrb)$ is well defined.

For the remainder of this paper, we fix $X$, a semi-algebraic subset of $\mbb{R}^n$.

\begin{dfn}
A \emph{semi-algebraic k chain in} $\mbb{R}^n$ is a current constructed as above. Denote by $S_k(\mbb{R}^n)$ the set of all such currents. We define

\begin{equation*}
S_k(X)=\{S\in S_k(\mbb{R}^n) \ | \ \spt(S)\subset X\}
\end{equation*}

\noindent to be \emph{the set of semi-algebraic $k$ chains of $X$}.

\end{dfn}

It is easy to see that $S_k(X)$ is an Abelian group.

The following notation will be used frequently throughout the remainder of the paper: for any smooth, compact oriented $k$-dimensional submanifold $N\subset \mbb{R}^n$ we denote by $\llb N \rrb$ the $k$-current given by $\llb N \rrb (\omega) = \int_N \phi^\ast (\omega)$, where $\phi$ denotes the inclusion $N\hookrightarrow \mbb{R}^n$.

If $\nu_N$ is an orientation form for $N$ the area formula gives:

\begin{equation}
\int_N \phi^\ast (\omega)=\int_N \omega(x)(\nu_N (x) ) \, d \mathcal{H}^k (x)
\end{equation}

\noindent Proposition \ref{prop:classical} will show that this causes no ambiguity with the notation $f_\sharp (\llb M \rrb)$ defined above.

Arguing with the constancy theorem of \cite[4.1.31]{GMT} one can show the following.

\begin{prop}[{\cite[3.5]{HardtHomotopy}}]
\label{prop:functoriality}
For $M$ and $f$ as above $\llb M \rrb \in S_k(\R^m)$, $\llb \partial M \rrb \in S_{k-1}(\R^m)$ and, further, $\partial(f_\sharp(\llb M \rrb))=f_\sharp ( \llb \partial M \rrb )$. In particular, since the boundary of a semi-algebraic submanifold is a finite, disjoint union of lower dimensional semi-algebraic submanifolds, the chain  $\partial (f_\sharp(\llb M \rrb ) ) \in S_{k-1}(\R^m)$.

\end{prop}

We can use this to classify semi-algebraic chains without the need for the maps $f$:

\begin{thm}[Alternative description of semi-algebraic chains  \cite{HardtHomotopy} ]
\label{thm:Altern}
If $V_1,...V_r$ are disjoint real-analytic semi-algebraic oriented $k$-dimensional submanifolds of $\mbb{R}^n$ such that each $\overline{V_i}$ is a compact subset of $X$, and $n_1,...n_r$ are integers, then:

\begin{equation*}
\ds\sum_{i=1}^r n_i \llb V_i \rrb \in S_k(X).
\end{equation*}

\noindent Conversely, every element of $S_k(X)$ admits such a representation.
\end{thm}

 Thus the set of real-analytic semi-algebraic oriented k-dimensional submanifolds generate the group $S_k(X)$ of semi-algebraic $k$ chains in $X$ just as the set of {\it oriented $k$ simplices} in $\R^n$  generate the group ${\mathcal P}_k(\R^n)$ of $k$ dimensional {\it polyhedral chains} in $\R^n$. Clearly ${\mathcal P}_k(\R^n)$ is a subgroup of $S_k(\R^n)$.

We now turn to functoriality of the semi-algebraic homology theory. To any semi-algebraic map $g: X \rightarrow Y$, we associate a map $g_\sharp: S_k(X)\rightarrow S_k(Y)$   by:

\begin{equation*}
g_\sharp \left( f_\sharp (\llb M \rrb ) \right)=(g\circ f)_{\sharp} (\llb M \rrb ).
\end{equation*}

\noindent To see that the rule $g_\sharp$ is well defined suppose $h_\sharp(\llb N \rrb)=f_\sharp (\llb M \rrb)$. Choose a stratification $\{M_i\}_{i=1}^m$ for $M$ and $\{N_j\}_{j=1}^\ell$ for $N$ compatible with the maps $f$ and $g \circ f$ (respectively, $h$ and $g \circ h$). Such stratifications may be obtained, for example, by refining stratifications compatible with $f$ (or $h$). For any $\omega \in \mathcal{D}^k (\mbb{R}^n)$:

\begin{equation*}
\ds \sum_{i=1}^m \int_{f(M_i)} \omega(x)(\nu_{f(M_i)}(x))d\mathcal{H}^k(x)=\ds \sum_{j=1}^\ell \int_{h(N_j)} \omega(x)(\nu_{h(N_j)}(x))d\mathcal{H}^k(x).
\end{equation*}

\noindent However, since each $h(n_j)$ and $f(M_i)$ have finite $\mathcal{H}^k$ measure this equality holds for any $\mathcal{H}^k$ summable $k$-form on $f(M)\cup h(N)$ by approximation. In particular, for $\omega \in \mathcal{D}^k(Y)$, $g^\ast(\omega)$ is such a summable $k$-form, which gives:

\begin{equation*}
\ds \sum_{i=1}^m \int_{f(M_i)} g^\ast(\omega)(x)(\nu_{f(M_i)}(x))d\mathcal{H}^k(x)=\ds \sum_{j=1}^\ell \int_{h(N_j)} g^{\ast}(\omega)(x)(\nu_{h(N_j)}(x))d\mathcal{H}^k(x)
\end{equation*}

\noindent However, by Equation \ref{eq:semichain} the left side of this equation is $(g\circ f)_\sharp (\llb M \rrb)$ and the right side is $(g\circ h)_\sharp (\llb N \rrb)$, so the rule $g_\sharp$ given above is well-defined.

The following proposition allows us to relate classical constructions using currents to semi-algebraic chains as defined above. We follow the notations for currents found in \cite[$\S$4]{GMT}.

In particular, for a Lipschitz map $f: X\rightarrow Y$, and a $k$-current of finite mass $T$, such that $f|_{\spt(T)}$ is proper, we denote the classical \emph{push forward} of the current $T$ (defined by smooth approximation in \cite[4.1.7]{GMT}) by $f_\sharp^{IC} (T)$.

\begin{prop}
\label{prop:classical}
For $S\in S_k(X)$, $A\subseteq X$ a semi-algebraic subset, if $Y\subset \mbb{R}^{m}$ semi-algebraic and  $f:X\rightarrow Y\subset \mbb{R}^m$ a smooth semi-algebraic map. Then

\begin{enumerate}
\item $S\restr A$ is a semi-algebraic chain.
\item The relation $f_{\sharp}(S)=f^{IC}_\sharp (S)$ holds whenever both sides are defined. That is, whenever $f$ is a Lipschitz semi-algebraic map such that $f|_{\spt(S)}$ is proper.
\item The current $\llb [0,1] \rrb \cross S$ is a semi-algebraic chain.

\end{enumerate}
\end{prop}

\begin{proof}
By linearity and Theorem \ref{thm:Altern}, we lose no generality in assuming that $S$ is of the form $\llbracket V \rrbracket$, where $V$ is a smooth semi-algebraic $k$ manifold in $\mbb{R}^n$ with compact closure. Let $\nu_V$ be the orientation form for $V$.

 Applying the stratification theorem \cite[9.1.8]{RAG} to the set $A$ and the singleton family $\{ V\cap A\}$, we get a stratification $\{A_i\}_{i=1}^p$ for $A$, where, for some $\ell \leq q$,  $V\cap A=\cup_{i=1}^\ell A_i$ a union of smooth, semi-algebraic strata. Relabeling if necessary, we remove from this union any stratum which is of dimension less than $k$. The remaining strata are smooth semi-algebraic manifolds of dimension $k$, and hence $k$-dimensional rectifiable sets. This implies that, for $\mathcal{H}^k$ a.e. $x\in A_i \cap V$, $\nu_V$ is an orientation form for $A_i$ (see, for example, \cite[2.85]{AFP}). Thus, for any $k$-form $\omega$:

\begin{equation*}
S\restr A (\omega)=\ds \int_{V\cap A} \omega(x)( \nu_V(x)) d \mathcal{H}^k(x)=\ds \sum_{i=1}^{\ell} \int_{A_i}\omega(x)(\nu_V(x)) d \mathcal{H}^k(x)=\sum_{i=1}^\ell \llb A_i \rrb (\omega)
\end{equation*}

\noindent and hence $S\restr A$ is a semi-algebraic chain, proving (1).

Since $V=\spt(S)$, we may stratify the map $f|_{\spt(S)}: \spt(S)\rightarrow Y$ as in \cite[2.5]{HardtHomotopy} to get a stratification of $V$ with each strata $A_i$ a smooth submanifold with $f|_{A_i}$ a diffeomorphism. By \cite[26.21]{Simon}, for any $\omega \in \mathcal{D}^k(\mbb{R}^m)$,

\begin{equation*}
f_\sharp^{IC} (S)(\omega)=\ds\int_V \omega(f(x))(Df_\sharp \nu_V(x)) d\mathcal{H}^k(x)=\ds\sum_{i=1}^n \int_{A_i}\omega(f(x))(Df_\sharp \nu_V(x)) d\mathcal{H}^k(x)
\end{equation*}

Since $f$ is a diffeomorphism on each $A_i$, each integrand in the last sum is

\begin{equation*}
\epsilon_i \int_{f(A_i)} \omega(x)(\nu_{A_i}(x)) d \mathcal{H}^k(x)
\end{equation*}

\noindent where $\epsilon_i=\pm 1$, depending on orientations. So, $f_\sharp^{IC} (S)=\sum_{i=1}^{n} \epsilon_i \llb f(A_i) \rrb$.

On the other hand, this semi-algebraic chain coincides with $f_\sharp (S)$ as defined in the previous section, giving (2).

Finally, (3) follows immediately by applying (1) to the set $A=[0,1]\cross \mbb{R}^n$ and the semi-algebraic chain $\llb (-1,2)\cross V \rrb$.

\end{proof}

\begin{rmk}
\label{rmk:homo}
The preceding Proposition allows us to generalize the well-known homotopy formula for currents to the semi-algebraic context. To this end, suppose that $M\subset \R^m$ is a smooth, oriented compact semi-algebraic manifold of dimension $k$ and $h: M \times [0,1] \rightarrow \R^n$ is a semi-algebraic map. Proposition \ref{prop:functoriality} guarantees that $\partial h_\sharp(\llb M \times [0,1] \rrb) =h_\sharp( \llb \partial \left( M\times [0,1] \right)\rrb)$. Consider the three sets $M\times \{0\}$, $M \times \{1\}$ and $\partial M \times [0,1]$. Since any two of these three sets intersect in a semi-algebraic set of dimension at most $k-2$, by stratifying the set $\partial (M\times [0,1])$ compatibly with the three sets listed above and their pairwise intersections, we see that Equation \ref{eq:semichain2} guarantees that the current $h_\sharp( \llb \partial\left( M\times [0,1]\right) \rrb)$ decomposes (up to sets of $\mathcal H^k$ measure zero) into $h_\sharp( \llb M \times \{1\} \rrb) - h_\sharp(\llb M \times \{0\} \rrb) + h_\sharp( \llb \partial M\times [0,1] \rrb)$ as in the classical case. 
\end{rmk}

\section{The Homology of Semi-Algebraic Chains}
\label{sec:Eilenberg}
By Proposition \ref{prop:functoriality} the semi-algebraic chains give a chain complex, $S_{\ast} (X)$. Similarly, for any semi-algebraic subset $A\subseteq X$, we can define a chain complex for the semi-algebraic pair $(X,A)$ by defining the Abelian group of relative cycles:

\begin{equation*}
\mathcal{Z}_k (X,A)=\{S \  | \ S\in S_k(X),   \ \partial S\in S_{k-1}(A)\}
\end{equation*}

\noindent (with the convention that if $k=0$ the condition $\partial S\in S_{k-1}(A)$ is trivially satisfied) and the subgroup of relative boundaries:

\begin{equation*}
\mathcal{B}_k (X,A)=\{S+\partial L \  | \ S\in S_k(A),  \ L\in S_{k+1}(X) \}.
\end{equation*}

\noindent we then define the $k$-dimensional semi-algebraic homology group of the pair $(X,A)$ as the quotient group:

\begin{equation*}
H_k^{SA}(X,A)=\mathcal{Z}_k(X,A) / \mathcal{B}_k(X,A).
\end{equation*}

\noindent As usual, in the case where $A=\phi$, we simply denote these sets as $\mathcal{Z}_k(X)$, $ \mathcal{B}_k(X)$ and $H_k^{SA}(X)$.

Working in the category of semi-algebraic pairs of sets and (continuous) semi-algebraic maps, we seek to show the following:

\begin{thm}
\label{thm:EilenbergSteenrod}
The homology groups associated to the chain complex of semi-algebraic chains on a semi-algebraic set $X$ satisfy the Eilenberg-Steenrod Axioms (1)-(7) listed below.
\end{thm}

 Moreover, by the compactness requirements of Theorem \ref{thm:Altern}, they also satisfy the axiom of compact support (see, for example, \cite[$\S$ 26]{Munkres}).

 This theorem is similar in spirit and method to the result proved in \cite[4.4.1]{GMT}, where it is shown that the homology groups associated to the chain complex of integral currents on a Lipschitz neighborhood retract satisfies the same axioms.

 To this end, let $(X,A)$ and $(Y,B)$ be such pairs and $f:(X,A)\rightarrow (Y,B)$ be a semi-algebraic map of pairs. By Proposition \ref{prop:functoriality} the map $f_\sharp$ induces to a map on the homology groups defined above; for convenience, also denote the induced map by $f_\sharp$. Similarly, for $A\subset X_0 \subset X$ we also let $\partial$ denote the induced map on homology $\partial : H_k(X,X_0) \rightarrow H_k(X_0,A)$. The Eilenberg-Steenrod axioms are as follows:

 \begin{enumerate}
\item $(id_{X})_\sharp$ is the identity map on $H_k(X,A)$.
\item If $g: (Y,B)\rightarrow (Z,C)$ is another admissible map, $(g\circ f)_\sharp=g_\sharp \circ f_\sharp$.
\item If $C \subset A \subset X$, $C' \subset B \subset Y$ and $f|_B: (A, C) \rightarrow (B, C')$ is admissible, then

\begin{equation*}
\left(f|_{A}\right)_{\sharp}\circ \partial = \partial \circ f_{\sharp}
\end{equation*}

\item If $(X_0,A)\xrightarrow{i} (X,A) \xrightarrow{j} (X, X_0)$ is a series of admissible inclusions, then:

\begin{center}
\begin{tikzpicture}[description/.style={fill=white,inner sep=2pt}]
    \matrix (m) [matrix of math nodes, row sep=3em,
    column sep=2.5em, text height=1.5ex, text depth=0.25ex]
    { H_k(X_0,A)  & H_k(X,A)  & H_k(X,X_0) \\
H_{k-1}(X_0,A) &  H_{k-1}(X,A)  & H_{k-1}(X,A_0) \\ };
       \path[->,font=\scriptsize]
    (m-1-1)  edge node[auto] {$ i_\sharp $} (m-1-2)
     (m-1-2)       edge node[auto] {$ j_\sharp $} (m-1-3)
    (m-2-1)  edge node[auto] {$ i_\sharp $} (m-2-2)
     (m-2-2)       edge node[auto] {$ j_\sharp $} (m-2-3)
    (m-1-3) edge node[above] {$ \partial $} (m-2-1);
\end{tikzpicture}
    \end{center}

    \noindent is exact for $k >0$, and, for $k=0$, $i_\sharp(H_0(X_0,A_0))=H_0(X,A)$.

\item if $h: I\cross X \rightarrow Y$ is an admissible homotopy between $f$ and $f':X\rightarrow Y$, then $f_\sharp=f'_\sharp$.

\item If $g:(X,A)\rightarrow (X',A')$ is the inclusion map and $\overline{(X'\setminus A')} \cap \overline{(X' \setminus X)} = \phi$, then $g_\sharp : H_k(X, A) \rightarrow H_k (X', A')$ is an isomorphism.

\item For any $a\in\mbb{R}^n$, $H_0(\{a\})=\mbb{Z}$ and, for $k>0$, $H_k(\{a\})=0$.

 \end{enumerate}
 \vspace{5 mm}

The proofs follow as in \cite[4.4.1]{GMT}, from various properties of currents and semi-algebraic chains:

 \vspace{5 mm}

\begin{proof}

(1)   follows from well-definedness of the map $id_\sharp : S_k(X) \rightarrow S_k(X)$. Indeed, $f_\sharp (\llb M\rrb)=(id_X \circ f)_\sharp (\llb M \rrb)$.

 (2) is functoriality of the map $f \mapsto f_\sharp$ and follows from the discussion in the previous section.

 By Proposition \ref{prop:functoriality} $\partial (f_{\sharp}(g_{\sharp} (\llb M \rrb))=f_{\sharp}(\partial g_{\sharp} (\llb M \rrb))$. If $f(A)\subseteq B$ and $\spt(\partial g_{\sharp}(\llb M \rrb)) \subset A$, then $\spt(f_{\sharp}\partial (g_{\sharp} (\llb M \rrb)))\subset B$. Equivilantly,

\begin{equation*}
\left(f|_{A}\right)_{\sharp}\circ \partial = \partial \circ f_{\sharp}
\end{equation*}

\noindent which shows (3).

Basic properties of currents give (4) upon noticing that the inclusion maps satisfy (2) of Proposition \ref{prop:classical}.

The homotopy formula for currents (extended to the semi-algebraic context as in Remark \ref{rmk:homo}) and part (3) of Proposition \ref{prop:classical} give (5).

Preceding the proof of (6), we make the following claim.  For a separate proof, see Proposition 2.2.8 of \cite{RAG}. 

\begin{claim} 
\label{claim:dist}
For any semi-algebraic $A\subset \R^n$, the function $\text{dist}(x,A)$ is semi-algebraic.
\end{claim}
\begin{proof}[Proof of Claim]
For simplicity, we may as well assume that $A$ is closed, since passing to the closure will not change the function in question. Let  $B=\{(x,r) \ | \ x\in \R^n \text{ and } \exists \ y\in A  \text{ with } ||x-y|| = r\}$. Note that $B$ is the projection of $\{(x,y,r) \ | \ x\in \R^n \ \ y\in A  \text{ and } ||x-y|| = r\}$, which is clearly semi-algebraic. 
 Next let $L$ be the connected component of the compliment of $B$ which contains $0\in \R^{n+1}$. This set is semi-algebraic, and, further, the frontier of $L$ is the graph of  $x \mapsto \text{dist}(x,A)$.

\end{proof}

 The proof of (6) then follows as in \cite[4.4.1]{GMT}  since the set:

\begin{equation*}
E=X'\cap \{x \ | \ \text{dist}(x,X'\setminus A') \leq \text{dist}(x, X'\setminus X)\}
\end{equation*}

\noindent is semi-algebraic. By repeatedly applying Proposition \ref{prop:classical}, all arguments follow exactly as stated.

For (7), note that the only non-empty semi-algebraic subset of the singleton set $\{a\}$ is itself--and hence the only allowable semi-algebraic chains are $n \llb \{a\} \rrb$ for $n\in \Z$, which implies the result.

\end{proof}

\section{Coincidence of Semi-Algebraic Homology and Integral Current Homology}
\label{sec:MT}

Our next goal is to show that the homology groups for semi-algebraic chains coincide with the homology groups of integral currents on a given semi-algebraic set $X$.

As in the introduction, let $ I_k(X)$ denote the Abelian group of $k$-dimensional integral currents, and $H^{IC}_k(X)$ denote the respective homology groups (see \cite[4.1.24]{GMT}). By comments in Section \ref{sec:elem} every semi-algebraic chain is necessarily an integral current which gives an inclusion $S_k(X)\overset{i}\hookrightarrow I_k(X)$. We will show that $i$ induces an isomorphism on the respective homology groups.

We first state a simplified version of the Deformation Theorem for integral currents:

\begin{thm}[{Deformation Theorem,  \cite[4.2.9]{GMT}}]
\label{thm:DFT}
For any integral current $T\in I_k(\mbb{R}^n)$ and $\epsilon >0$, there exists a polyhedral chain $P\in \mathcal P_k (\R^n)$, and integral currents  $Q\in  I_k(\R^n)$ and $L \in  I _{k+1}(\R^n)$ such that $T=P+Q+\partial L$, with $\mbox{\emph{spt}}(P)\cup \mbox{\emph{spt}}(L)\cup \mbox{\emph{spt}}(Q) \subset \{x \ |  \ \mbox{\emph{dist}}(x,\mbox{\emph{spt}}(T)) \leq 2n\epsilon\}$. 
Further, there exists a constant $\kappa$ so that:

\begin{enumerate}
\item $\mass(P) \leq \kappa ( \mass(T) + \epsilon \mass(\partial T) )$
\item $\mass(\partial P) \leq \kappa \mass(\partial T)$
\item $\mass(Q) \leq \epsilon \kappa \mass(\partial T)$
\item $\mass(S) \leq \epsilon \kappa \mass(T)$
\end{enumerate}

Finally, given a countable collection $\{T_j\}_{j\in \mbb N}$ of integral currents, we may set $T_j = P_j +Q_j +\partial L_j$ so that each $P_j$ is a finite sum of integer multiples of $k$ dimensional cubes from the same standard cubical decomposition of $\mathbb R^n$ with edge-length $\epsilon$.

\end{thm}

That we may choose the same cubical decomposition as the support for a countable collection of integral currents is crucial for generalizing mass minimization. While not explicitely stated in the theorem statement in \cite{GMT}, it follows easily from the proof, where it is shown that for any integral current $T$, almost any translation of the standard cubical decomposition may be taken as the support for $P$ in the above decomposition. 

As stated, this version of Theorem \ref{thm:DFT} is insufficient for our purposes. We need the additional proposition:

\begin{prop}
\label{prop:defo}
In Theorem \ref{thm:DFT}, if $\partial T$ is a semi-algebraic chain, then $Q$ may be chosen to be semi-algebraic as well.
\end{prop}

The above proposition follows from analyzing the definition of $Q$ in the proof of Theorem \ref{thm:DFT} in \cite[4.2.9]{GMT}; we will follow the definitions and notations given there. There are many known proofs of the Deformation Theorem (see, for example, the work of B. White \cite{whitedefo}), however, the proof presented in \cite{GMT} is very constructive and allows for a fairly simple proof of the above proposition.

The proof of Theorem \ref{thm:DFT} in \cite{GMT} produces $Q$ as a deformation of $\partial T$, and so $Q$ will be semi-algebraic inasmuch as $\partial T$ and the deformations are semi-algebraic.

\begin{proof}
$Q$ is defined to be a sum of deformations of $\partial T$ so it suffices to show that any element of this sum is semi-algebraic.

Elements of the sum are of the form:

\begin{equation*}
\ds \lim_{r\mapsto 0^+} h_\sharp^i (\llb I \rrb \cross \partial T\restr U_r). 
\end{equation*}

\noindent By transversality arguments, we may assume that the chosen $a$ has the additional property that $\spt(\partial T)\cap (W''_{n-m-1}+a)$ is empty. Therefore, by shifting the skeletons by $a$ we may drop the $\tau_a$ terms from our analysis and,  since $\spt(\partial T)$ is compact disregard the limit in the definition and consider only the chains $h_\sharp^i (\llb I \rrb \cross \partial T)$.

By Proposition \ref{prop:classical}, it suffices to show that $h^i$ is a semi-algebraic map. This would follow immediately from showing that the retractions $\sigma_i$ are semi-algebraic--however, a weaker condition is sufficient: since any compact semi-algebraic set (in particular, $\spt(\partial T)$) may intersect only finitely many cubes of the decomposition, all we must show is that the retraction maps are semi-algebraic on a given cube.

Given such a cube $C$ with center $q$, take $\mbb{A}=\{\Delta \ | \ \Delta \text{ is an $i$-face of } C\}$ and define for any $i$-face $\Delta$:
\begin{equation*}
S_{\Delta}=\{(x,y,t)\in (C\setminus W''_{n-i-1}) \cross \Delta \cross \mbb{R} \ | \ x=q+t(y-q), t>0\}
\end{equation*}
\noindent Note that $S_\Delta$ is semi-algebraic.

Using the geometric description of $\sigma_i$ given in \cite[4.2.6]{GMT}, the graph of $\sigma_i$ over the cube $C$ is given by the finite union

\begin{equation*}
\ds\bigcup_{\Delta \in \mbb{A}} \Pi_{x,y}\left(S_{\Delta}\right)
\end{equation*}

\noindent where $\Pi_{x,y}$ denotes the projection onto the first components. Since this union is finite, the Tarski-Seidenberg principle \cite[1.4]{RAG} gives that the graph is semi-algebraic. Therefore $\sigma_i |_{K}$ is semi-algebraic on any such cube, completing the proof.

\end{proof}

The next remark will allow us to reduce to the compact case.

\begin{rmk}
\label{rmk: compact}
Given any compact set $A\subset X$, there exists a compact semi-algebraic set $B$ such that $A \subset B \subset X$. To see this, note that since $A$ is compact $A\subset B_r(0)$ for some $r\geq 0$. Further, $A\cap \Fron(X)=\phi$ and moreover, there is an $\epsilon > 0$ so that $\text{dist}(x,\Fron(X)) >\epsilon$ on $A$. Since, by Claim \ref{claim:dist}, $f(x)=\text{dist}(x,\Fron(X))$ is semi-algebraic (see the comment in the proof of (6) of Theorem \ref{thm:EilenbergSteenrod}), so too is the closed set
\begin{center}
$f^{-1}([\tfrac{\epsilon}{2},\infty])\cap X$.
\end{center}
\noindent So, the semi-algebraic set
\begin{center}
$B=\left(f^{-1}([\tfrac{\epsilon}{2},\infty])\cap X\right)\cap \overline{B_r(0)}$
\end{center}
\noindent is a compact subset of $X$ with the desired properties.

Thus, for any finite collection of currents $T_1$, $T_2$,...,$T_n$, by applying the above to $A=\bigcup \spt(T_i)$ and replacing $X$ by $B$, we may do all of our computations in $B$.
\end{rmk}

Our main result is to show the following:

\begin{thm}
\label{thm:main}
 The inclusion map $i:S_k(X,A) \rightarrow  I_k(X,A)$ induces an isomorphism on the respective homology groups $i_\ast: H^{SA}_k(X,A) \rightarrow H_k^{IC}(X,A)$.
  \end{thm}

We will prove this result first in the special case where $A=\phi$ and then derive the general result from this special case. We first prove injectivity of the induced map. Surjectivity relies on this fact, and will be proven shortly.

\begin{prop}
\label{prop:inj}
 The inclusion map $i:S_k(X) \rightarrow  I_k(X)$ induces an injection on the respective homology groups $i_\ast: H^{SA}_k(X) \rightarrow H_k^{IC}(X)$.
  \end{prop}
  
The proof has two steps--the first is to show that bounded subsets of $X$ are semi-algebraic retractions of certain arbitrarily small open neighborhoods--this is simply due to the triangulability of any compact semi-algebraic subset of $\R^n$. The second is to notice that in any open set, the Deformation Theorem and Proposition \ref{prop:defo} give that any integral current with semi-algebraic boundary is homologous to a semi-algebraic chain, and that semi-algebraic chains may be pushed forward through the above retraction.

\begin{proof}[Proof of Proposition \ref{prop:inj}]

It is clear that $i$ commutes with the boundary operator, and hence extends to a map $i_\ast$ on the associated homology groups

To prove the proposition, it suffices to show that if $S$ is a semi-algebraic chain supported on $X$ which bounds an integral current supported on $X$, then $S$ also bounds a semi-algebraic chain on $X$. To this end, let $S=\partial T$, where $T\in I_{k+1}(X)$. Applying Remark \ref{rmk: compact} to the currents $S$ and $T$, we may assume that $X$ is compact.
Enclose $X$ within the interior of a cube $C$. Triangulate the semi-algebraic set $C$ compatibly with the closed semi-algebraic subset $X$ (as in \cite{RAG}) to obtain a simplicial complex $K$ and a semi-algebraic homeomorphism $h: |K|\rightarrow C$. Since $X$ is closed, compatibility implies that $X$ is a sub-complex of the triangulation--that is, that there exists a subcomplex $L$ of $K$ such that $h|_{|L|} : |L|\rightarrow X$ is a semi-algebraic homeomorphism. Via (repeated) barycentric subdivision (up to relabeling) we may assume that $L$ is a full subcomplex of $K$, and hence that there exists an open neighborhood (the so-called regular neighborhood) $N(L)\subset |K|$ of $|L|$ on which when any $\alpha \in N(L)$ is written in barycentric coordinates as a (non-negative) linear combination of vertices from $K$:
\begin{equation*}
\alpha = \ds \sum_{a \text{ is a vertex of } K} \alpha(a) a
\end{equation*}  

\noindent we have that

\begin{equation*}
\ds \sum_{a \text{ is a vertex of } L}  \alpha(a) > 0.
\end{equation*}

\noindent This enables us to define a retraction map $r: N(L)\rightarrow X$ given as follows:

\begin{equation*}
r(\alpha)=\dfrac{\ds \sum_{a \text{ is a vertex of } L}  \alpha(a)a}{\ds \sum_{a \text{ is a vertex of } L}  \alpha(a)}.
\end{equation*}

\noindent For more details, see \cite[II, $\S$ 9]{FAT}.  Since $|L|$ is semi-algebraic and $r$ is piecewise linear, we obtain that $r$ is semi-algebraic on $|L|$, and hence by composing with the semi-algebraic triangulating homeomorphism, there is a semi-algebraic retraction $\td{r}:h(N(L))\rightarrow X$. Since $h(N(L))$ is open in $C$, by restricting to $E=h(N(L))\cap \mathring{C}$ we may assume that $X$ is a Euclidean neighborhood retract in $\mbb{R}^n$. Further, since $X$ is compactly contained in $E$, $E$ contains some $\delta$ neighborhood of $X$, and so for $\epsilon >0$ small enough, we may apply Theorem \ref{thm:DFT} and Proposition \ref{prop:defo} to get a polyhedral $P$,  and integral currents $Q$ and $R$ in $S_{k+1}(E)$ and $I_{k+2}(E)$, respectively, such that $T=P+Q+\partial R$. Applying the semi-algebraic retraction to both sides of $\partial T = \partial P +\partial Q=\partial(P+Q)$ gives $S=\partial T=\td r_{\sharp} (\partial T)=\td r_{\sharp} (\partial(P+Q))=\partial \td r_{\sharp}(P+Q)$, where the last equality is justified by Proposition \ref{prop:functoriality} since all terms involved are semi-algebraic chains. Thus $S$ bounds the semi-algebraic chain $\td{r}_{\sharp}(P+Q)$, which proves the injectivity of $i_\ast$.

\end{proof}

Surjectivity requires another property of semi-algebraic sets which is a corollary of Theorem 5.1 of \cite{deRham}. A simplified version is given below, although it should be noted that the statement in \cite{deRham} is more general and stronger. 

\begin{thm}{\cite[Theorem 5.1]{deRham}}
Let $X_1,...,X_k$ be a stratification for $X=\cup X_i\subset \R^n$, and let $p\in X$. 
\begin{enumerate}
\item There exists a neighborhood $U$ of $p$ in $X$ and a stratification $U_1,...,U_\ell$ of $U$ so that each $X_i \cap U$ is a union of some subcollection of $U_1,...,U_\ell$.
\item There exists a semi-algebraic $N\subset U$ with $p\in N$ and  $\dim(N) < \dim(U)$ and a Lipschitz strong deformation retraction $r: U \times [0,1] \rightarrow U$ to $N$ such that:
\begin{enumerate}
\item $r(x,0) \in N$ and $r(x,1) = x$ for $x\in X$.
\item For any $j$, $r(U_j)\cross(0,1])\subset U_j$.
\end{enumerate}
\end{enumerate}
\end{thm}

\noindent We can then prove the following corollary by inducting on the dimension of the set $N$ obtained in the previous theorem. For a more elementary proof, see Theorem 4.1.5 of \cite{Shartserthesis}.

\begin{corr}
\label{thm:retractions}
Let $X_1,\dots,X_m$ be semi-algebraic subsets of $\R^m$ such that $X=\cup_{j=1}^m X_j \subset \mbb{R}^n$ is closed. Further, suppose that $0\in \overline{X_j}\cap X$ for all $j$. Then there exists a neighborhood $U$ of $0$ in $\mbb{R}^n$ and a semi-algebraic Lipschitz deformation retraction to $0$, $r:U\times I\rightarrow U $ that preserves the $X_j$'s.
\end{corr}

To assist in the proof, we make the following remark.

\begin{rmk}
\label{rmk:sblip}
If $h: X \rightarrow Y$ is semi-algebraic and a homeomorphism, then its inverse is also semi-algebraic, since the graph of $h$, $\{(x,y)\in X\cross Y \ | \ y=h(x)\}=\{(x,y)\in X\cross Y \ | \ h^{-1}(y)=x\}$ is a coordinate permutation of the graph of $h^{-1}$, and thus the graph of $h$ is semi-algebraic if and only if the graph of $h^{-1}$ is. For brevity, we call such a map a semi-algebraic homeomorphism. If we further suppose that $h$ is bilipschitz, then the following diagram commutes:

\begin{center}
\begin{tikzpicture}[description/.style={fill=white,inner sep=2pt}]
    \matrix (m) [matrix of math nodes, row sep=3em,
    column sep=2.5em, text height=1.5ex, text depth=0.25ex]
    { H^{SA}_\ast(X)  & H^{IC}_\ast(X)\\
H^{SA}_{\ast}(Y) &  H^{IC}_{\ast}(Y) \\ };
       \path[->,font=\scriptsize]
    (m-1-1)  edge node[auto] {$ i_\ast $} (m-1-2)
    (m-2-1)  edge node[auto] {$ i_\sharp $} (m-2-2);

       \path[<->,font=\scriptsize]
     (m-1-1)       edge node[auto] {$ h_\ast $} (m-2-1)
	(m-1-2) edge node[auto] {$ h^{IC}_\sharp $} (m-2-2);
\end{tikzpicture}
    \end{center}
\noindent The vertical arrows are isomorphisms since $f$ is invertible in both the Lipschitz and semi-algebraic categories, so the bijectivity of the inclusion induced map $i_\ast: H^{SA}_k(X) \rightarrow H_k^{IC}(X)$ is invariant under such transformations.

\end{rmk}

\begin{prop}
\label{prop:surj}
 If $X\subset \mbb{R}^n$ is semi-algebraic then the inclusion map $i:S_k(X) \hookrightarrow  I_k(X)$ induces a surjection on the respective homology groups $i_\ast: H^{SA}_k(X) \rightarrow H_k^{IC}(X)$.
  \end{prop}

The proof will require basic knowledge of slicing of currents; the aspects that we require are documented below. The main reference for the theory of slicing is \cite{FST}. In particular, the following Proposition comprises a part of \cite{FST}, Corollary 3.6.

\begin{prop}\label{prop:slice}
If $f : \R^n \rightarrow \R$ is Lipschitz and $T\in I_k(\R^n)$ has $\partial T=0$. Then, for almost every $y\in \R$, there is a current $<T,f,y> \in I_{k-1}(\R^n)$ so that:

\begin{enumerate}
\item $\mbox{\emph{spt}}(<T,f,y>) \subseteq \mbox{\emph{spt}}(T) \cap f^{-1}(y)$
\item $<T,f,y> = \partial(T\restr \{f>y\} )$
\end{enumerate}

\end{prop}

\noindent The current $<T,f,y>$ is referred to as the \emph{slice of $T$ in $f^{-1}(y)$}.

We will also require the following consequence of Hardt's trivialization theorem--for details, see \cite[\S 9.3]{RAG}. 

\begin{thm}[{Local Conic Structure,  \cite[Theorem 9.3.6]{RAG}}]
\label{thm:LCS}
Let $A\subset \R^n$ be semi-algebraic and $x$ a nonisolated point of $E$. There exists $\epsilon > 0$ and a semi-algebraic homeomorphism $\varphi: \overline{B_\epsilon (x)} \rightarrow \overline{B_\epsilon (x)}$ so that:
\begin{enumerate}
\item $||\varphi(y) - x || = ||y-x||$ for all $y\in\overline B_\epsilon(x)$,
\item $\varphi|_{\partial B_\epsilon(x)}$ is the identity mapping.
\item $\varphi^{-1} ( A \cap \overline B_\epsilon (x) ) $ is the cone with vertex $x$ and basis $A \cap \partial B_\epsilon(x)$. 
\end{enumerate}
\end{thm}

 The following notation will be used heavily in the proof:

\begin{itemize}
\item For $y>0$, denote by $ C(x_0,y)$ the cube with faces perpendicular to the coordinate axes, side length $2y$ and centered at $x_0\in\mbb{R}^n$
\item For $x,z\in\mbb{R}^n$ with components $x_i$ and $z_i$, respectively, let $D(x,z)=\ds\max_{i=1,...n}|x_i-z_i|$.
\item For $x_0\in \mbb{R}^n$, denote by $D_{x_0}:\mbb{R}^n\rightarrow \mbb{R}$ the function $D_{x_0}(y)=D(y,x_0)$.
\item For $x_0\in\mbb{R}^n$ and $y>0$ and $T\in I_k(X)$, denote by $<T,D_{x_0},y>$  the slice of $T$ at $y$ by the Lipschitz function $D_{x_0}$, as defined in Proposition \ref{prop:slice}.
\end{itemize}

The main idea of the proof is induction on the ambient dimension. For the inductive step, we choose some integral curent $T \in I_k(X)$ with $\partial T=0$ and  we consider the intersection of  $X$ with the elements of a (finite) cover of X by cubes whose boundaries slice the current $T$ well, intersect$X$ with strictly smaller dimension, and inside of which $X$ is contractible. Roughly speaking, for a fixed cube $C$ of the decomposition, we then apply the inductive assumption to the slice of $T$ in the boundary of $C$ and use the contraction to construct a semi-algebraic current within $C$ whose boundary is contained in the boundary of $C$ and is homologous in the boundary of $C$ to the boundary of $T\restr (X\setminus C)$, then we use injectivity of $i_\ast$ (Proposition \ref{prop:inj}) to restart the process in $X\setminus C$, which is covered by fewer cubes. Iterating this process and keeping track of the changes we make in each cube then gives the result.

\begin{proof}[Proof of Proposition \ref{prop:surj}]

We first prove surjectivity in the case that $\dim(X)\leq n-1$, and we prove the statement by induction on the ambient dimension. If $X\subset \mbb{R}$, $X$ is a finite union of points, and the proof of the proposition is obvious.

Now suppose $X\subset \mbb{R}^n$, and choose $T\in I_k(X)$ with $\partial T=0$. By Remark \ref{rmk: compact} applied to $A=\spt(T)$ we may assume $X$ is compact. The following claim helps us to decompose $X$.

\begin{claim}
If $x_0\in X$, then there exists an $\epsilon > 0$ so that, for  $0<\delta < \epsilon$, $\dim\left(\partial C(x_0,\delta)\cap X\right)\leq n-2$. 
\end{claim}

\begin{proof}[Proof of Claim] First note that if $x_0$ is an isolated point of $X$, there is nothing to prove. If $x_0$ is nonisolated, by the coarea formula, if $C_1 d^X (D_{x_0})_y$ denotes the 1 dimensional coarea factor of the tangential differential of the Lipschitz function $D_{x_0}$ on the $H^{n-1}$ rectifiable set $X$ at the point $y$, then:

\begin{equation*}
 \int_0^1 \mathcal H^{n-2} (X\cap \partial C(x_0,s)) \, ds = \int_{X\cap C(x_0,1)} C_1 d^X (D_{x_0})_y d \, \mathcal H^{n-1}(y) 
 \end{equation*}
 
 \noindent (see, for example, Theorem 2.93 of \cite{AFP}). However, we may directly compute that $C_1 d^X (D_{x_0})_y< M$ for some constant $M$, which gives: 
 
 \begin{equation*}
\int_0^1 \mathcal H^{n-2} (X\cap \partial C(x_0,s)) \, ds \leq M \mathcal H^{n-1}(X\cap C(x_0,1)) < \infty
 \end{equation*}
   
 \noindent Therefore, for a.e. $s\in[0,1]$, $\mathcal H^{n-2}(X\cap \partial C(x_0,s)) < \infty$.  In particular, there are arbitrarily small $\delta$ so that  $\mathcal H^{n-2}(X\cap\partial C(x_0,\delta) ) < \infty$. However,  for any such $\delta$, $X \cap \partial C(x_0,\delta)$ is semi-algebraic and therefore the Hausdorff measure bound  implies  $\dim\left(\partial C(x_0,\delta)\cap X\right)\leq n-2$.  Finally, Theorem \ref{thm:LCS} guarantees that, for $\epsilon >0 $ small enough and $0 < \delta < \epsilon$ the dimension of the sets, $X \cap \partial C(x_0,\delta)$ is constant, and this gives the claim.
\end{proof}

  Invoking Corollary \ref{thm:retractions}, the compactness of $X$ and the above claim, we find a finite set of cubes $\{(C(x_i,y_i)\}_{i=1}^\ell$ such that:

\begin{enumerate}
\item For all $i$, $x_i\in X$ and $y_i > 0$.
\item For every $i$, there exists a semi-algebraic, Lipschitz deformation retraction to the point $x_i$,  $f_i:C(x_i,y_i)\cross I\rightarrow U(x_i,y_i)$ which preserves $X$, where $U(x_i,y_i)$ is some open neighborhood of $C(x_i,y_i)$.
\item For every $i$, $y_i$ is such that the slice $<T,D_{x_i},y_i>\in  I_{k-1}(X)$ exists.
\item $\dim\left(\partial C(x_i,y_i)\cap X\right)\leq n-2$.
\end{enumerate}

\noindent We now ``push off'' of each cube individually. If $X$ can be contained in just one cube then it is (semi-algebraically) Lipschitz contractible by Corollary \ref{thm:retractions}, and there is nothing to prove, so we assume in the following that $\ell > 1$ and that

\begin{center}
$X\subset \ds\bigcup_{i=1}^\ell C(x_i,y_i)$.
\end{center}

 \noindent Consider the slice $<T,D_{x_1},y_1>$. By Proposition \ref{prop:slice} (1), we have:

\begin{center}
$\spt(<T,D_{x_1},y_1>) \subset \{x\in\mbb{R}^n \ | D(x,x_1)=y_1\}\cap X$
\end{center}
\noindent and so  $<T,D_{x_1},y_1>$ is supported in
\begin{center}
$\partial C(x_1,y_1)\cap X $.
\end{center}

\noindent Further, Proposition \ref{prop:slice} (2) guarantees that
\begin{center}
 $<T,D_{x_1},y_1> = \partial(T \restr \{D_{x_1} > y_1\})$;
\end{center}
\noindent so we have that $\partial<T,D_{x_1},y_1>=0$.

However, $\partial C(x_1,y_1)\cap X $  is, by assumption, a semi-algebraic set of dimension at most $n-2$. By stratifying the cube compatibly with this intersection we may find a point $a_1\in \partial C(x_1,y_1)$ and an $\epsilon>0$ such that $B(a_1,\epsilon)\cap \partial(C(x_1,y_1)) \cap X = \phi$. Since $\partial C(x_1,y_1) \setminus B(a_1,\epsilon)$ is bilipschitz semi-algebraically homeomorphic to a subset of $\mbb{R}^{n-1}$. Noting Remark \ref{rmk:sblip} and our inductive hypothesis, we get:

 \begin{center}
 $<T,D_{x_1},y_1>=L_1+\partial F_1$
 \end{center} for $L_1\in S_{k-1}(X\cap \partial C(x_1,y_1))$ and $F_1\in I_{k}(X\cap \partial C(x_1,y_1))$.

Applying the homotopy formula we find:

\begin{center}
$\partial \left( T\restr C(x_1,y_1) - (f_1)_\sharp (L_1\cross \llb0,1\rrb) - F_1\right)=0$
\end{center}

\noindent  since $C(x_1,y_1)$ is contractible in a way that preserves $X$, we have that

\begin{center}
$T\restr C(x_1,y_1) - (f_1)_\sharp (L_1\cross \llb0,1\rrb) - F_1=\partial K_1$ \end{center}

\noindent for some $K_1\in I_{k+1}(X\cap C(x_1,y_1) )$.

Denote by $X_j=X\cap\left(\ds\bigcup_{k=j+1}^\ell C(x_k,y_k)\right)$.  Adding $T$ to both sides of the above and rearranging gives:

\begin{equation}
\label{eq:important}
T\restr( X\setminus C(x_1,y_1)) =T-\partial K_1-(f_1)_\sharp (L_1\cross \llb0,1\rrb) -F_1
\end{equation}

\noindent We now shift our focus to

\begin{center}
$T\restr (X\setminus C(x_1,y_1)) +F_1\in I_{k+1}(X_1 )$.
\end{center}

\noindent Equation \ref{eq:important} yields that:

\begin{center}
$-\partial \left((T\restr X\setminus C(x_1,y_1)) +F_1\right)=\partial \left((f_1)_\sharp (L_1\cross \llb0,1\rrb)\right)$.
\end{center}

\noindent Notice that $\partial \left((f_1)_\sharp (L_1\cross \llb0,1\rrb)\right)\in S_{k-1}(X_1)$. So, by Proposition \ref{prop:inj}, since $\partial \left((f_1)_\sharp (L_1\cross \llb0,1\rrb)\right)$ is the boundary of a rectifiable current in $I_{k}(X_1 )$ there is a semi-algebraic chain $S_1\in S_k(X_1)$ such that:
\begin{center}
$\partial S_1=-\partial \left((T\restr X\setminus C(x_1,y_1)) +F_1\right)$
\end{center}

Then, $S_1+(T\restr X\setminus C(x_1,y_1)) +F_1$ is a rectifiable cycle supported in $X_1$. Using Equation \ref{eq:important}, we get:

\begin{center}
$S_1+(T-\partial K_1-(f_1)_\sharp (L_1\cross \llb0,1\rrb) -F_1)+F_1$
\end{center}

\noindent is a cycle of ${I}_k(X_1)$. Applying the above discussion to this current and the semi-algebraic set $X_1$ gives:

\begin{equation}
\label{eq:step2}
S_2+\left(S_1+T-\partial K_1-(f_1)_\sharp (L_1\cross \llb0,1\rrb)\right)\restr (X_1\setminus C(x_2,y_2) )+F_2
\end{equation}

\noindent is an element of ${I}_k(X_2)$, for $S_2\in S_{k}(X_2)$ and $F_2\in I_{k}(X\cap \partial C(x_2,y_2))$.

On the other hand, in analogy with Equation \ref{eq:important} applied to:

\begin{center}
 $\left(S_1+T-\partial K_1-(f_1)_\sharp (L_1\cross \llb0,1\rrb)\right)\restr (X_1\setminus C(x_2,y_2))$
 \end{center}

 \noindent we may find  $L_2\in S_{k-1}(X_1\cap \partial C(x_2,y_2))$  and a $K_2\in I_{k+1}(X_1\cap \partial C(x_2,y_2) )$ so that:

\begin{equation}
\begin{array}{c}
(S_1+T-\partial K_1-(f_1)_\sharp (L_1\cross \llb0,1\rrb))\restr (X_1\setminus C(x_2,y_2))=\\
(S_1+T-\partial K_1-(f_1)_\sharp (L_1\cross \llb0,1\rrb))-\partial K_2-(f_2)_\sharp (L_2\cross \llb0,1\rrb) -F_2.
\end{array}
\end{equation}

\noindent Plugging this into Equation \ref{eq:step2}, we get that:

\begin{center}
$\begin{array}{c}
S_2+S_1+T-\partial K_1 - (f_1)_\sharp (L_1\cross \llb0,1\rrb)\\
 - \partial K_2 - (f_2)_\sharp (L_2\cross \llb0,1\rrb)
\end{array}$
\end{center}

\noindent is a cycle of ${I}_k(X_2)$. Continuing in this way, we get:

\begin{equation}
\label{eq:final}
T+\ds\sum_{i=1}^{\ell-1} \left(S_1-\partial K_1 - (f_i)_\sharp (L_i\cross \llb0,1\rrb)\right)
\end{equation}

\noindent is a cycle of ${I}_k(X_\ell)$, where  $F_i\in I_{k}(X\cap \partial C(x_i,y_i))$, $S_i\in S_{k}(X_i)$, $L_i\in S_{k-1}(X_i\cap \partial C(x_i,y_i))$. However, $X_\ell=X\cap C(x_\ell,y_\ell) $ is Lipschitz semi-algebraically contractible in a way which preserves $X$,  so we get:

\begin{center}
$T+\ds\sum_{i=1}^{\ell-1} \left(S_1-\partial K_1 - (f_i)_\sharp (L_i\cross \llb0,1\rrb)\right)=\partial K_\ell$
\end{center}

\noindent for some $K_\ell \in {I}_{k+1}(X_\ell)$. Rearranging terms gives:

\begin{center}
$T-\ds\sum_{i=1}^{\ell} \partial K_i=-\sum_{i=1}^{\ell-1} \left(S_1- (f_i)_\sharp (L_i\cross \llb0,1\rrb)\right)$.
\end{center}

\noindent  However, all the terms on the right hand side are semi-algebraic, so we have proved the proposition when $X \subset \R^n$ and $\dim(X) \leq n-1$.

The general case follows immediately from this special case--since any semi-algebraic set may be embedded in a higher dimensional space by simply adding on dimensions. More precisely, $X$ is semi-algebraically bilipschitz equivalent to $X\times \{0\}\subset \mbb{R}^{n+1}$.

\end{proof}

We now show that $i_\ast$ inducing an isomorphism between $H^{SA}_k(X)$ and $H^{IC}_k(X)$ implies the seemingly stronger result that $i_\ast$ induces an isomorphism on the relative homology groups, i.e. $H_k^{SA}(X,A)\cong H_k^{IC}(X,A)$, and this is Theorem \ref{thm:main}.

The proof has essentially the same two parts--for each $[T]\in H_k^{IC} (X,A)$, we need to find a semi-algebraic chain $S\in \mathcal{Z}_k(X,A)\cap [T]$, and, if $S \in [0] \in H_k^{IC}(X,A)$, we need to find  $Q\in S_{k+1}(X)$ and $L \in S_k(X)$ so that $S=L+\partial Q$.

\begin{proof}[Proof of Theorem \ref{thm:main}]
First, let $[T] \in H_k^{IC}(X,A)$ be given. Then $\spt(\partial T) \subset A$ and so $\partial T \in I_{k-1}(A)$.  Applying Proposition \ref{prop:surj} to the homology class $[\partial T] \in H_{k-1}^{IC} (A)$, we know that $\partial T = S_1 + \partial K_1$ for $S_1 \in S_{k-1} (A)$ and $K_1 \in I_k (A)$.

 As an element of $S_{k-1}(X)$, $S_1$ is the boundary of the rectifiable current $T-K_1$. So, by Proposition \ref{prop:inj} we know $S_1 = \partial S_2$ for $S_2 \in S_{k}(X)$. But then $\partial (T-K_1 - S_2)=0$ so again by Proposition \ref{prop:surj} we find $S_3\in S_k(X)$ and $K_2 \in I_{k+1} (X)$ so that $(T-K_1-S_2)=S_3 + \partial K_2$. Rearranging this gives $T-(S_2+S_3)= K_1 + \partial K_2$, which is to say that $(S_2 + S_3) \in [T]$.

 Next suppose that $S \in [0] \in H_k^{IC}(X,A)$. This is equivalent to saying $\spt(\partial(S))\subset A$ and the existence of $T \in I_k(A)$ and $F\in I_{k+1}(X)$ so that $S=T+\partial F$. Then, as an element of $I_k(A)$, $\partial S = \partial T$, and hence by Proposition \ref{prop:inj} we may find a semi-algebraic $S_1 \in S_k(A)$ so that $\partial S =\partial S_1$ (note that in general we could not have chosen $S_1=S$ since $S$ need not be supported in $A$).  Then, $\partial (T-S_1) = 0$ so by Proposition \ref{prop:surj} applied within $A$, we may find $S_2 \in S_{k+1}(A)$ and $K \in I_{k+1} (A)$ so that $T- S_1 = S_2 + \partial K$. Then:

 \begin{equation*}
 S-S_1 = (T-S_1) + \partial F = (S_2 + \partial K) + \partial F = S_2 + \partial ( K+F)
 \end{equation*}

 \noindent rearranging this equation gives:

 \begin{equation*}
 S-S_1-S_2 = \partial(K+F)
 \end{equation*}

 \noindent and so the semi-algebraic chain $S-S_1-S_2$ bounds a rectifiable current in $X$, so applying Proposition \ref{prop:inj} gives $S-S_1-S_2=\partial S_3$ for some $S_3 \in S_{k+1} (X)$ . Therefore, $S=(S_1 + S_2) + \partial (S_3)$, completing the proof.

\end{proof}

While the above isomorphism holds in all cases, if $X$ is compact and $A\subset X$ is closed, then  by semi-algebraicity, the pair $(X,A)$  is  triangulable as a finite simplicial complex. Further, one may check that the proof of the uniqueness theorem, Theorem 10.1  of  \cite[III]{FAT}  holds provided the admissible category includes all simplicial complexes embedded in Euclidean space and all simplicial maps from them to other elements of the category--the point is that the main isomorphism for the uniqueness result is defined by axiomatically computing  homology groups for very basic simplicial complexes which may be assumed to be embedded in high dimensional space and pushing forward generators of these groups.  This fact, combined with the results of Section \ref{sec:Eilenberg} proves the following:

\begin{corr}
If $(X,A)$ is a compact pair of semi-algebraic sets, then the homology groups $H^{IC}_\ast (X,A)$ are isomorphic to the ordinary singular homology groups of the pair $(X,A)$.
\end{corr}


\section{Mass Minimization in Homology Classes}

As mentioned in the introduction, our main result gives mass minimization in the homology classes of semi-algebraic sets. More precisely, we have the following:

\begin{thm}
For a compact semi-algebraic set $X\subset \R^m$ and $[S]\in H_k(X)$ a semi-algebraic homology class there exists a rectifiable representative $T\in [S]$ with
\begin{equation*}
 \mbb{M}(T) = \ds\inf_{L\in[S]} \mbb{M}(L).
 \end{equation*}

\end{thm}

\begin{proof}
Let $\td r :X_\delta\rightarrow X$ denote the retraction from some $\delta$ neighborhood of $X$ contained in $h(N(L))$ constructed in the proof of Proposition \ref{prop:inj}, and choose a minimizing sequence $\{T_\ell\}_{\ell=1}^\infty\subset I_{k}(X)$ such that:
\begin{enumerate}
\item $T_\ell \in [S]$, i.e. there is a sequence $\{R_\ell\}_{\ell=1}^{\infty}\subset I_{k+1}$ so that $S-T_\ell = \bd R_\ell$ 
\item $\mass(T_\ell) \rightarrow \ds \inf_{L\in[S]} \mbb{M}(L)$.
\end{enumerate}

\noindent Let $\epsilon < \tfrac{\delta}{2n}$. Since $\{R_\ell\}_{\ell=1}^{\infty}$ is countable, we may find a cubical decomposition of $\R^m$ which is in good position for every $R_\ell$ of side length $\epsilon$ so that all cubes intersecting $X$ are contained in $X_\delta$. By compactness, there are only finitely many cubes of the decomposition which intersect $X_{\frac{\delta}{2}}$; denote by $\{A_i\}_{i=1}^q$ and $\{B_j\}_{j=1}^p$ their $k+1$ and $k$ faces, respectively.

Applying the Deformation Theorem (Theorem \ref{thm:DFT}) to each $R_\ell$ gives: $R_\ell = P_\ell + Q_\ell +\partial L_\ell$ so that:

\begin{enumerate}
\item $P_\ell, \ Q_\ell$ and $L_\ell$ are supported in $X_\delta$ for every $\ell$
\item $\mass(Q) \leq \epsilon \kappa \mass(\partial R_\ell) =\epsilon \kappa \mass(S-T_\ell) \leq 2 \epsilon \kappa \mass(S)$
\item $\mass(\partial P_\ell) < \kappa \mass(\partial R_\ell) = \kappa \mass(S-T_\ell)$
\item $P_\ell$ is a sum of the form $\ds \sum_{i=1}^q \alpha(i,\ell) \llb A_i \rrb $.
\end{enumerate}

\noindent This gives $S-T_\ell = \partial P_\ell + \partial Q_\ell$. So, since $\{\mass(\partial P_\ell)\}_{\ell=1}^\infty$ is bounded, the collection $\{\mass(\partial Q_\ell)\}_{\ell=1}^\infty$ is also bounded, and therefore we may apply the well known compactness theorem for integral currents to obtain (up to relabeling) weakly convergent subsequences and currents $T\in I_k(X)$, $Q\in I_{k+1}(X_\delta)$ so that $T_\ell \warrow T$, $Q_\ell \warrow Q$.

To apply a similar compactness result to the collection $\{P_\ell\}_{\ell=1}^{\infty}$, we need to uniformly bound the mass of the collection. Since each $P_\ell$ is supported on the same finite polyhedral skeleton, this amounts to proving bounds on the coefficients $\alpha(i, \ell)$. Let $V$ be the vector space generated by the $A_i$'s and $W$ be the vector space generated by the $B_i$'s. Then the boundary operator may be represented by a linear map $D : V \rightarrow W$ whose entries are all $\pm 1$ or 0 and, since $V$ and $W$ are finite dimensional is a bounded operator.

Put $\partial P_\ell = \sum_{j=1}^p \beta(j,\ell) \llb B_j \rrb$ and $\vec{\beta}_\ell=(\beta(1,\ell) , \cdots , \beta(p, \ell) )$. Then, for each $\ell \in \mbb{N}$, we have a linear system $D \vec{x} = \vec{\beta}_\ell$. However, since the mass of $\partial P_\ell$ is uniformly bounded, $|\vec{\beta}_\ell  |$ is uniformly bounded. Since each $\vec\beta_\ell$ is integer valued, this gives only finitely many such vectors. Finally, by assumption, for each $\ell$ we have a solution to the linear system $D \vec{x} = \vec{\beta}_\ell$, but since there are only finitely many such systems, by choosing a particular solution for each one we immediately get density bounds for some polyhedral sequence $F_\ell \in I_{k+1}(X_\delta)$ with $\partial F_\ell = \partial P_\ell$.  Compactness gives $F\in I_{k+1}(X_\delta)$ so that $F_\ell \warrow F$. This gives:

\begin{equation*}
S-T = \partial F +\partial Q.
\end{equation*}

\noindent However, this implies that $S-T$ is also null-homologous in $I_k(X)$ since, applying Proposition \ref{prop:surj} in $X_\delta$ we get $[S-T]=[L]=[0]$ for $L\in S_k(X_\delta)$. Proposition \ref{prop:inj} then guarantees $L=\partial B$ for some $B\in S_{k+1}(X_\delta)$ Applying the retraction (in $I_k(X)$) gives $[S-T]=[\td r_{\sharp} (L)]=[\td r_{\sharp} (\partial B)] = [\partial \td r_{\sharp} (B) ] = [0]$. The proof is completed upon referencing the lower semicontinuity of mass with regards to weak convergence.

\end{proof}


\section{Homological Triviality of Small Cycles}

Finally, we use our main theorem to show that, in the compact semi-algebraic case, we can guarantee that all rectifiable cycles of sufficiently small mass bound. This follows from an analogous result for Lipschitz neighborhood retracts, shown in \cite{FFCurrents}.

\begin{prop}
For any compact semi-algebraic set $X\subset \mbb{R}^n$, there is an $\epsilon > 0$ so that if $T \in I_k(X)$ has $\bd T =0$ and $\mass(T) \leq \epsilon$, then there is a $F \in I_{k+1}(X)$ so that $T=\bd F$.
\end{prop}

\begin{proof}
Let $\td r: X_\delta \rightarrow X$ denote the neighborhood of $X$ and the retraction constructed in the proof of Proposition \ref{prop:inj}. By \cite[9.6 (3)]{FFCurrents} there is an $\epsilon > 0$ so that the proposition holds for $X_\delta$. Let $T \in I_k(X)$  be as above with this choice of $\epsilon$.

By Proposition \ref{prop:surj}, there is a semi-algebraic $S \in [T]$. However, in $X_\delta$, $T=\partial F$, so, if $S-T = \partial K$, then $S=\partial (K+F)$. Applying Proposition \ref{prop:inj} to $S$ in $X_\delta$, we find $S=\partial L$ for $L \in S_{k+1} (X_\delta)$. However, then $[T]=[S]= [\partial \td r_{\sharp} (L)]=[0]$, completing the proof.
\end{proof}



\bibliographystyle{plain} \bibliography{reference}

\end{document}